\documentclass[12pt]{article}

\usepackage{amsmath,amssymb,amsthm,tikz}

\newcommand{\ZZ}{\mathbb Z}
\newcommand{\RR}{\mathbb R}
\newcommand{\Homeo}{\operatorname{Homeo}}
\newcommand{\rk}{\operatorname{rk}}

\newtheorem*{theorem}{Theorem}

\title{Totally disconnected groups (not) acting on two-manifolds}

\author{John Pardon\thanks{This research was conducted during the period the author served as a Clay Research Fellow.  The author was also partially supported by a Packard Fellowship and by the National Science Foundation under the Alan T.\ Waterman Award, Grant No.\ 1747553.}}

\date{6 July 2018}

\begin{document}

\maketitle

\begin{abstract}
We briefly survey the Hilbert--Smith Conjecture, and we include a proof of it in dimension two (where it is originally due to Montgomery--Zippin).
\end{abstract}

Hilbert's Fifth Problem \cite{hilbert,hilbertbams} asks for an investigation into the extent to which assumptions of real analyticity in the theory of Lie group actions on manifolds can be relaxed to mere continuity.
This question admits a number of different precise formulations, which currently range from completely solved to wide open.
Our brief discussion is far from complete; we refer the reader to Tao \cite{tao} and Palais \cite{noticesgleason} for more details.

One precise formulation of Hilbert's Fifth Problem is the question of whether every topological manifold $G$ with continuous group operations $m:G\times G\to G$ and $i:G\to G$ has a unique real analytic structure for which the group operations are real analytic (in which case $G$ is called a Lie group).
This question was answered positively by Gleason \cite{gleason2} and Montgomery--Zippin \cite{montgomeryzippin}; see also the later work of Hirschfeld \cite{hirschfeld} and Goldbring \cite{goldbring}.
In fact, the work of Gleason \cite{gleason} and Yamabe \cite{yamabe1,yamabe2} supplies a stronger result, known as the Gleason--Yamabe theorem, which states that every locally compact group $G$ has an open subgroup $G_0\subseteq G$ which is an inverse limit of Lie groups (we follow the convention whereby topological groups are always taken to be Hausdorff).
A topological group is said to be NSS (or, to have no small subgroups) iff there exists a neighborhood of the identity in $G$ which contains no nontrivial subgroup.
A Lie group is obviously NSS (use the exponential map), and it follows immediately from the Gleason--Yamabe theorem that a locally compact group $G$ is Lie iff it is NSS (note, however, that proving the implication NSS $\implies$ Lie is an important step in the proof of the Gleason--Yamabe theorem).
This is an amazing rigidity result, characterizing Lie groups among all locally compact groups as those satisfying a simple condition combining the group structure and the topology.

The story for groups acting on manifolds (rather than groups which are themselves manifolds) is less settled.
The main open problem is the Hilbert--Smith Conjecture, which asserts that every locally compact group acting faithfully on a connected manifold $M$ is a Lie group.
By a faithful action of $G$ on $M$, we simply mean a continuous injective homomorphism $G\to\Homeo(M)$, where $\Homeo(M)$ is given the compact-open topology (i.e.\ the topology of uniform convergence on compact subsets).
Connectivity of $M$ (which could equivalently be relaxed to having finitely many connected components) is necessary to avoid examples such as the compact group $\prod_{n\geq 0}\ZZ/2^n$ acting on $\bigsqcup_{n\geq 0}S^1$ by rotation in each factor.
The hypothesis that $G$ be locally compact may seem somewhat artificial, though clearly some hypothesis is needed to exclude $G=\Homeo(M)$, which is certainly not a Lie group.
Also, note that even when $G$ is a Lie group, the action of $G$ on $M$ can be rather pathological (as in not locally conjugate to a smooth or even a piecewise linear action).
For example, Bing constructed an involution of $S^3$ whose fixed set is a wildly embedded two-sphere \cite{binginvolution} as well as a non-manifold $X^3$ with the property that $\RR\times X^3=\RR^4$ \cite{bingproduct} (thus giving rise to a non-smoothable action of $S^1$ on the manifold $S^1\times X^3$).

Since every group acts faithfully on itself by translation, the Hilbert--Smith Conjecture implies the positive answer to the version of Hilbert's Fifth Problem discussed above.
In the other direction, the classification of locally compact groups (the Gleason--Yamabe theorem) together with Newman's theorem, to be discussed shortly, allows one to reduce the Hilbert--Smith Conjecture to the special case of $G=A_p=\ZZ_p:=\varprojlim\ZZ/p^n$, the additive group of $p$-adic integers (see Tao \cite{tao} or Lee \cite{lee}).
Note that the group $\ZZ_p$ is self-similar: the subgroups $p^k\ZZ_p\subseteq\ZZ_p$ (which form a neighborhood base at the identity as $k\to\infty$) are all isomorphic to $\ZZ_p$.
In studying the Hilbert--Smith Conjecture, we may therefore, by replacing the given action of $\ZZ_p$ by the action of $p^k\ZZ_p$ for some sufficiently large $k$, assume without loss of generality that the given action of $\ZZ_p$ is $C^0$-close to the trivial action.
One reason to expect this observation to be helpful is that, by the Gleason--Yamabe theorem, these ``small subgroups'' are precisely the salient feature of the group $\ZZ_p$ which distinguishes it from a Lie group.
It is also interesting to compare this observation with the result of Newman \cite{newman} (see also Smith \cite{smith} and Dress \cite{dress}) that nontrivial actions of compact Lie groups on manifolds cannot be arbitrarily $C^0$-close to the trivial action.

\begin{theorem}[Newman \cite{newman}]
Let $M$ be a connected manifold equipped with a metric $d$, and let $U\subseteq M$ be a nonempty open subset.
There exists $\varepsilon=\varepsilon(U,d|_U)>0$ such that for every compact Lie group $G$, the only action of $G$ on $M$ satisfying $d(x,gx)\leq\varepsilon$ for all $x\in U$ and $g\in G$ is the trivial action.
In particular, if $G$ acts trivially on $U$, then it acts trivially on all of $M$.
\end{theorem}

(Note the importance of assuming $G$ is compact: a nontrivial $\RR$-action which is arbitrarily $C^0$-close to the identity can be constructed on any manifold $M$ by integrating any nonzero vector field supported inside a small ball.)

\begin{proof}
It is enough to treat the case $G=\ZZ/p$ for any prime $p$.
Indeed, the kernel of $G\to\Homeo(M)$ is a closed normal subgroup $H\trianglelefteq G$, so $G/H$ is a Lie group and thus, if nontrivial, contains a subgroup isomorphic to $\ZZ/p$ for some prime $p$.
Note that the $\varepsilon>0$ we produce must be independent of $p$.

We deduce the desired result from the following two claims about $\ZZ/p$ actions on open subsets of $\RR^n$.
Suppose $\ZZ/p$ acts on $V\subseteq\RR^n$ and that $V$ contains the unit ball centered at the origin $0\in\RR^n$.
\begin{itemize}
\item If $d(x,gx)\leq\frac 12$ for every $x\in V$ and $g\in\ZZ/p$, then $\ZZ/p$ acts trivially on a neighborhood of the origin.
\item If $\ZZ/p$ acts trivially on $V\cap(\RR_{\geq 0}\times\RR^{n-1})$, then $\ZZ/p$ acts trivially on a neighborhood of the origin.
\end{itemize}
To deduce the desired result, note that the first claim produces an $\varepsilon=\varepsilon(U,d|_U)>0$ such that any $G$ action on $M$ satisfying $d(x,gx)\leq\varepsilon$ for all $x\in U$ acts trivially on some nonempty open subset $U_0\subseteq U$, and the second claim implies that a $G$ action on $M$ acting trivially on a nonempty open subset is globally trivial.

To prove the first claim, we consider the averaging map
\begin{align*}
A:V&\to\RR^n,\\
x&\mapsto\frac 1p\sum_{g\in\ZZ/p}gx.
\end{align*}
(I have had less than resounding success in eliciting a laugh by remarking that this proof does not work in characteristic $p$.)
The hypothesis implies that $d(x,A(x))\leq\frac 12$.
It follows that $A$ is proper over a neighborhood of $0\in\RR^n$, and hence its degree over $0\in\RR^n$ is well-defined.
In fact, the inequality $d(x,A(x))\leq\frac 12$ implies the degree of $A$ over $0\in\RR^n$ equals $1$, using the obvious linear homotopy between $A$ and the identity map.
On the other hand, if any point $x\in\RR^n$ close to the origin is not fixed by the $\ZZ/p$ action, then the factorization $V\to V/(\ZZ/p)\to\RR^n$ implies that the degree of $A$ over $x\in\RR^n$ is divisible by $p$ (note that $A^{-1}(x)\subseteq V$ is compact and contained in the open locus where $\ZZ/p$ acts freely).
We thus conclude that $\ZZ/p$ acts trivially on a neighborhood of the origin, as desired.

To prove the second claim, we again consider the degree of the averaging map near the origin.
First, we must shrink $V$ to $(\ZZ/p)(\RR_{>-\delta}\times\{x\in\RR^{n-1}:\left|x\right|<\frac 12\})$ for sufficiently small $\delta>0$; this ensures that $A$ is proper over a neighborhood of $0\in\RR^n$, so the degree is well-defined (and locally constant) there.
Again, the degree of $A$ near the origin equals $1$, this time by considering the degree over nearby points in the interior of the halfspace $\RR_{\geq 0}\times\RR^{n-1}$.
Now just as before, if any point close to the origin were not fixed by the $\ZZ/p$ action, then the degree of $A$ over that point would be divisible by $p$, a contradiction.
\end{proof}

It is, of course, tempting to try to adapt the proof of Newman's theorem to treat $\ZZ_p$ actions as well and thus prove the Hilbert--Smith Conjecture.

There are a number of partial results towards the Hilbert--Smith Conjecture.
\begin{itemize}
\item Yang \cite{ctyang} showed that for any faithful action of $\ZZ_p$ on an $n$-manifold $M$, the orbit space $M/\mathbb Z_p$ must have cohomological dimension $n+2$.
\item For $C^2$ actions, the Hilbert--Smith Conjecture is due to Bochner--Montgomery \cite{bochnermontgomery}.
\item For $C^{0,1}$ actions, the Hilbert--Smith Conjecture is due to Repov{\u{s}}--{\u{S}}{\u{c}}epin \cite{hslipschitz}.
\item For $C^{0,\frac n{n+2}+\varepsilon}$ actions, the Hilbert--Smith Conjecture is due to Maleshich \cite{hsholder}.
\item For quasiconformal actions, the Hilbert--Smith Conjecture is due to Martin \cite{hsqc}.
\item For uniformly quasisymmetric actions on doubling Ahlfors regular compact metric measure manifolds with Hausdorff dimension in $[1,n+2)$, the Hilbert--Smith Conjecture is due to Mj \cite{mj}.
\item For transitive actions, the Hilbert--Smith Conjecture is due to Montgomery--Zippin \cite[p243]{montgomeryzippin2} (see also Tao \cite[\S 1.6.4]{tao}).
\item When $\dim M=1$, the Hilbert--Smith Conjecture is an easy exercise.
\item When $\dim M=2$, the Hilbert--Smith Conjecture is due to Montgomery--Zippin \cite[p249]{montgomeryzippin2}.
\item When $\dim M=3$, the Hilbert--Smith Conjecture is due to the author \cite{pardon}.
\end{itemize}
By work of Raymond--Williams \cite{raymondwilliams}, there are faithful actions of $\mathbb Z_p$ on $n$-dimensional compact metric spaces which achieve the cohomological dimension jump of Yang \cite{ctyang} for every $n\geq 2$.
We also remark that by work of Walsh \cite[p282 Corollary 5.15.1]{walsh} there \emph{does} exist a continuous decomposition of any compact PL $n$-manifold into cantor sets of arbitrarily small diameter if $n\geq 3$ (see also Wilson \cite[Theorem 3]{wilson}).

To conclude this short survey, we present a proof of the Hilbert--Smith Conjecture in dimension two.
The proof we present is substantially different from the proof given by Montgomery--Zippin \cite[p249]{montgomeryzippin2}, which is more point set topological.
We were led to the proof presented below by trying to adapt the proof from \cite{pardon} to two dimensions.

\begin{theorem}[Hilbert--Smith Conjecture in dimension two]
Any locally compact group admitting a faithful action on a connected two-dimensional manifold is a Lie group.
\end{theorem}

\begin{proof}
By the reductions discussed above, it suffices to show that $\ZZ_p$ does not admit a faithful action on a connected two-dimensional manifold $M$ for any prime $p$ (in fact, our argument below applies to any profinite group in place of $\ZZ_p$).
We suppose the existence of such an action and work to derive a contradiction.

Fix a point $x\in M$ such that no subgroup $p^k\ZZ_p$ acts trivially on an open neighborhood of $x$.
The existence of such a point $x$ is guaranteed by Newman's theorem.
Indeed, suppose that, on the contrary, for every $x\in M$ there exists a $k<\infty$ such that $p^k\ZZ_p$ acts trivially on an open neighborhood of $x$.
It follows that for every compact subset $K\subseteq M$, there exists a $k<\infty$ such that $p^k\ZZ_p$ acts trivially on $K$.
Fix some $k_0<\infty$ such that $p^{k_0}\ZZ_p$ acts trivially on some nonempty open subset $U_0$ of $M$.
Now for every pre-compact connected open subset $U\subseteq M$ intersecting $U_0$, we know that $p^k\ZZ_p$ acts trivially on $U$ for some $k<\infty$.
But now applying Newman's theorem to the action of the finite group $p^{k_0}\ZZ_p/p^k\ZZ_p$ on $U$, we conclude that in fact $p^{k_0}\ZZ_p$ acts trivially on $U$.
Since $U$ was arbitrary, we conclude that $p^{k_0}\ZZ_p$ acts trivially on all of $M$, contradicting faithfulness.
This proves the existence of the desired point $x\in M$.

We will derive a contradiction by analyzing the $\ZZ_p$ action locally near $x$.

Fix a closed disk $D\subseteq M$ around $x$, and identify $D$ with the standard unit disk in $\RR^2$ with $x$ as its center.
By replacing $\ZZ_p$ with $p^k\ZZ_p$ for sufficiently large $k$, we assume that the action is extremely close to the trivial action over $D$ (though note that, of course, $\ZZ_p$ need not stabilize $D$).

\begin{figure}[htb]
\centering
\includegraphics{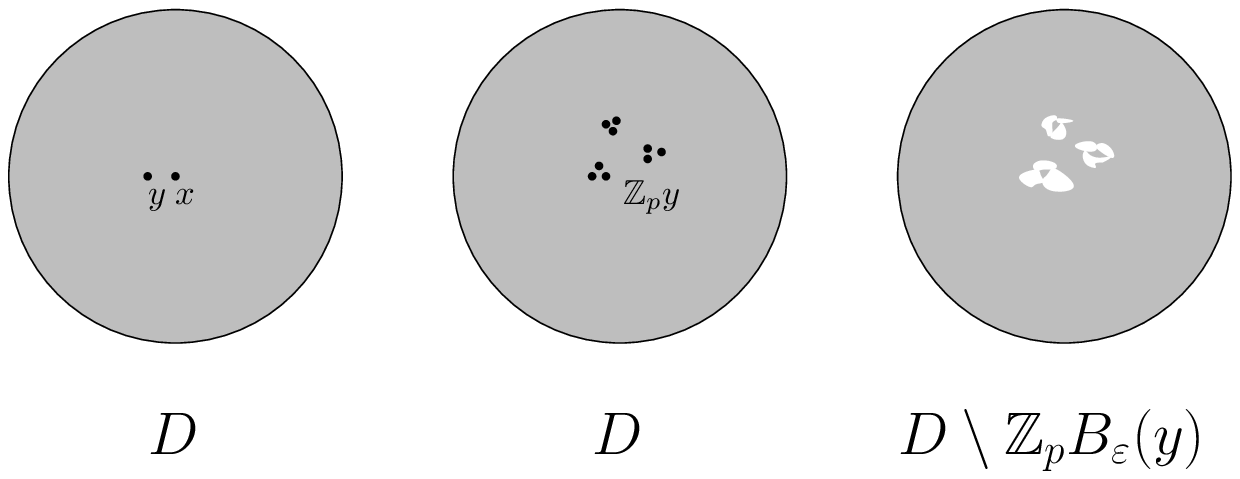}
\end{figure}

Fix a point $y\in D$ close to $x$ which is not fixed by $\ZZ_p$.
The orbit $\ZZ_py$ is thus either a cantor set or $p^k$ points for some $k\geq 1$ (depending on the stabilizer of $y$); in particular, $\rk H_1(D\setminus\ZZ_py)\geq 2$.
Let $Q$ denote the large component of $D\setminus\ZZ_pB_\varepsilon(y)$ (where $B_\varepsilon(y)$ denotes the closed $\varepsilon$-ball centered at $y$).
By choosing $\varepsilon>0$ sufficiently small, we can ensure that $\rk H_1(Q)\geq 2$ as well (by a direct limit argument).
Now $\ZZ_pB_\varepsilon(y)$ has finitely many connected components (since $\ZZ_p$ acts transitively on them and the stabilizer of the component containing $y$ is open), and hence $Q$ is finitely connected.
It follows that $Q$ is \emph{homeomorphic to the complement of $n\geq 2$ points in the closed disk}.

\begin{figure}[htb]
\centering
\includegraphics{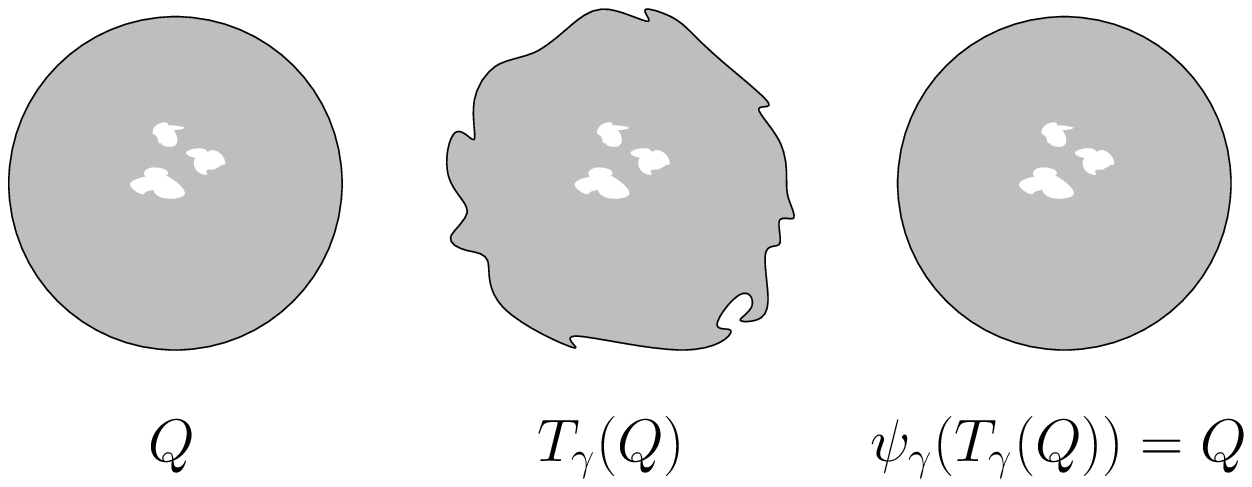}
\end{figure}

We now argue that the action of $\ZZ_p$ on $Q$ gives rise to a homomorphism to the braid group
\begin{equation*}
\ZZ_p\to B_n,
\end{equation*}
where $n\geq 2$ is the number of ``holes'' in $Q$.
Of course, $\ZZ_p$ does not truly act on $Q$, rather it sends the outer boundary $\partial D$ to a nearby circle.
Nevertheless, we may still produce a homomorphism to the braid group by observing that, since the action $T_\gamma:Q\to M$ of any $\gamma\in\ZZ_p$ is very close to the identity, there is a canonical (up to homotopy) map $\psi_\gamma:M\to M$ supported near $\partial D$ such that $\psi_\gamma\circ T_\gamma$ restricts to the identity map on $\partial D$.
Now the map $\ZZ_p\to B_n$ sends $\gamma$ to the class $[\psi_\gamma\circ T_\gamma]\in\pi_0\Homeo(Q,\partial D)=B_n$, and one can check that this map is indeed a group homomorphism.

Now the homomorphism $\ZZ_p\to B_n$ is continuous, where the target $B_n$ is equipped with the discrete topology, and hence its image is a finite subgroup of $B_n$.
Since $\ZZ_p$ acts transitively on the finitely many components of $\ZZ_pB_\varepsilon(y)$, this finite subgroup of $B_n$ must be nontrivial (as it \emph{a fortiori} must have transitive image in the symmetric group $S_n$ and $n\geq 2$).
On the other hand, the braid group $B_n$ is well-known to have no nontrivial finite subgroups (see \cite[\S 2]{braid} for example).
This contradiction finishes the proof.
\end{proof}

\textbf{Acknowledgements:} I would like to thank the referee for their excellent suggestions for improving the clarity of the proofs.

\bibliographystyle{amsplain}
\bibliography{hs2d}

\end{document}